\documentclass[12pt]{amsart}
\usepackage{amsmath}
\usepackage{amsthm}
\usepackage{amsfonts}
\usepackage{amssymb}
\usepackage{mathrsfs}
\usepackage{graphicx}
\usepackage{stmaryrd}
\usepackage{dsfont}
\usepackage[all]{xy}
\usepackage[shortlabels]{enumitem}
\usepackage[dvipsnames]{xcolor}
\usepackage[margin=1in]{geometry} 
\usepackage[bookmarks, bookmarksdepth=2, colorlinks=true, linkcolor=blue, citecolor=blue, urlcolor=blue]{hyperref}
\usepackage{eucal}
\usepackage{contour}
\usepackage[normalem]{ulem}
\usepackage{tikz}
\usetikzlibrary{decorations.markings}

\makeatletter
\@addtoreset{equation}{section}
\makeatother

\numberwithin{equation}{section}
\newtheorem{theorem}[equation]{Theorem}

\newtheorem{proposition}[equation]{Proposition}
\newtheorem{lemma}[equation]{Lemma}
\newtheorem{corollary}[equation]{Corollary}

\theoremstyle{definition}
\newtheorem{rmk}[equation]{Remark}
\newenvironment{remark}[1][]{\begin{rmk}[#1] \pushQED{\qed}}{\popQED \end{rmk}}
\newtheorem{eg}[equation]{Example}
\newenvironment{example}[1][]{\begin{eg}[#1] \pushQED{\qed}}{\popQED \end{eg}}
\newtheorem{defnaux}[equation]{Definition}
\newenvironment{definition}[1][]{\begin{defnaux}[#1]\pushQED{\qed}}{\popQED \end{defnaux}}

\newcommand{\arxiv}[1]{\href{http://arxiv.org/abs/#1}{{\tiny\tt arXiv:#1}}}
\newcommand{\DOI}[1]{\href{http://doi.org/#1}{\color{purple}{\tiny\tt DOI:#1}}}

\newcommand{\sB}{\mathscr{B}}
\newcommand{\bC}{\mathbf{C}}
\newcommand{\cC}{\mathcal{C}}

\newcommand{\sE}{\mathscr{E}}

\newcommand{\bL}{\mathbf{L}}

\newcommand{\rL}{\mathrm{L}}

\newcommand{\cM}{\mathcal{M}}

\newcommand{\bN}{\mathbf{N}}

\newcommand{\bQ}{\mathbf{Q}}

\newcommand{\bR}{\mathbf{R}}

\newcommand{\bS}{\mathbf{S}}
\newcommand{\cS}{\mathcal{S}}
\newcommand{\fS}{\mathfrak{S}}

\newcommand{\sT}{\mathscr{T}}

\newcommand{\bZ}{\mathbf{Z}}

\newcommand{\ra}{\mathrm{a}}

\newcommand{\rb}{\mathrm{b}}

\newcommand{\rc}{\mathrm{c}}

\newcommand{\rd}{\mathrm{d}}

\newcommand{\rf}{\mathrm{f}}

\newcommand{\fp}{\mathfrak{p}}

\let\ol\overline
\let\ul\underline
\let\lbb\llbracket
\let\rbb\rrbracket

\newcommand{\defn}[1]{\emph{#1}}
\renewcommand{\phi}{\varphi}
\renewcommand{\emptyset}{\varnothing}

\DeclareMathOperator{\End}{End}

\DeclareMathOperator{\Aut}{Aut}

\DeclareMathOperator{\Hom}{Hom}

\DeclareMathOperator{\Rep}{Rep}

\newcommand{\id}{\mathrm{id}}

\newcommand{\bone}{\mathbf{1}}
\newcommand{\bbone}{\mathds{1}}

\contourlength{1pt}

\contourlength{0.8pt}
\newcommand{\myuline}[1]{%
  \uline{\phantom{#1}}%
  \llap{\contour{white}{#1}}%
}
\DeclareMathOperator{\uRep}{\text{\myuline{\rm Rep}}}
\DeclareMathOperator{\uPerm}{\ul{Perm}}

\title{Measures for the colored circle}

\author{Andrew Snowden}
\address{Department of Mathematics, University of Michigan, Ann Arbor, MI}
\email{\href{mailto:asnowden@umich.edu}{asnowden@umich.edu}}
\urladdr{\url{http://www-personal.umich.edu/~asnowden/}}

\date{February 16, 2023}

\begin{document}

\begin{abstract}
In recent work with Harman, we introduced a new notion of measure for oligomorphic groups, and showed how they can be used to produce interesting tensor categories. Determining the measures for an oligomorphic group is (in our view) an important and difficult combinatorial problem, which has only been solved in a handful of cases. The purpose of this paper is to solve this problem for a certain infinite family of oligomorphic groups, namely, the automorphism group of the $n$-colored circle (for each $n \ge 1$).
\end{abstract}

\maketitle
\tableofcontents

\section{Introduction}

\subsection{Background}

Suppose $G$ is an algebraic group (or supergroup) over a field $k$. One can then consider the category $\Rep(G)$ of finite dimensional algebraic representations of $G$, which comes with a tensor product. This category satisfies the following conditions:
\begin{enumerate}
\item It is abelian and every object has finite length (i.e., a finite composition series).
\item The space of maps between two objects is a finite dimensional $k$-vector space.
\item Every object has a dual (i.e., the category is rigid).
\item If $\bbone$ denotes the unit object for tensor product (i.e., the trivial representation) then $\End(\bbone)=k$.
\end{enumerate}
A \defn{pre-Tannakian category} is a $k$-linear symmetric tensor category satisfying these axioms (see \cite[\S 2.1]{ComesOstrik} for more details). An important problem within the field of tensor categories is to understand the extent to which pre-Tannakian categories go beyond classical representation categories.

Deligne \cite{Deligne} gave the first examples of pre-Tannakian categories not of the form $\Rep(G)$: he constructed a 1-parameter family of pre-Tannakian categories $\uRep(\fS_t)$ by ``interpolating'' the representation categories $\Rep(\fS_n)$ of symmetric groups. Knop \cite{Knop, Knop2} generalized Deligne's construction and interpolated other families of finite groups, such as finite linear groups. There has been much subsequent work in this direction, e.g., \cite{ComesOstrik1, ComesOstrik, CW, EntovaAizenbudHeidersdorf, Harman, Harman2}.

Recently, in joint work with Harman \cite{repst}, we gave a new construction of pre-Tannakian categories. Recall that an \defn{oligomorphic group} is a permutation group $(G, \Omega)$ such that $G$ has finitely many orbits on $\Omega^n$ for all $n \ge 0$. We introduced a notion of \defn{measure} for an oligomorphic group (reviewed in \S \ref{s:olig} below). Given a $k$-valued measure $\mu$ for $G$, we constructed a $k$-linear rigid tensor category $\uPerm(G; \mu)$ of ``permutation modules.'' Under certain hypotheses, we showed that this category admits an abelian envelope $\uRep(G; \mu)$ that is pre-Tannakian.

The simplest example of an oligomorphic group is the infinite symmetric group $\fS$. In this case, we showed in \cite{repst} that there is a 1-parameter family of measures $\mu_t$, and that the resulting category $\uRep(\fS; \mu_t)$ coincides with Deligne's interpolation category $\uRep(\fS_t)$. More generally, in all known cases of interpolation (such as those considered by Knop), the sequence of finite groups has an oligomorphic limit, and our theory yields the interpolation categories previously constructed.

In \cite{repst}, we also considered a handful of oligomorphic groups that do not arise as limits of finite groups. For example, we analyzed the the oligomorphic group $\Aut(\bR,<)$ of order-preserving self-bijections of the real line. We showed that this group admits essentially four measures, and that one of these measures leads to a pre-Tannakian category. This category was studied in detail in \cite{delannoy}, where it was named the \defn{Delannoy category}; we found that it possess several remarkable properties (e.g., the Adams operations are trivial on its Grothendieck group). A related case will be treated in the forthcoming paper \cite{delannoy2}.

The purpose of this paper is to add to the list of examples from \cite{repst}: we determine the measures for a certain infinite family of oligomorphic groups (which cannot be realized as limits of finite groups). This leads to a large number of new rigid tensor categories. We do not know if these categories have abelian envelopes.

\begin{remark}
If $\cC$ is a class of finite relational structures (such as graphs, total orders, etc.), one can sometimes form the \defn{Fra\"iss\'e limit} $\Omega$ of $\cC$, which is a countable structure that has an important homogeneity property. The automorphism group $G$ of $\Omega$ is often oligomorphic, and this construction is the main source of oligomorphic groups. We showed in \cite{repst} that a measure for $G$ is (essentially) a rule assigning to each member of $\cC$ a value in $k$ such that certain identities hold. Thus understanding measures is really a combinatorial problem, and indeed, most of the work in this paper is combinatorially in nature.
\end{remark}

\subsection{Statement of results}

Let $\bS$ be a countable set equipped with an everywhere dense cyclic order; for example, one can take $\bS$ to be the roots of unity in the complex unit circle. Let $\Sigma$ be a non-empty finite set and let $\sigma \colon \bS \to \Sigma$ be a function such that $\sigma^{-1}(a)$ is dense for each $a \in \Sigma$. We regard $\sigma$ as a coloring of $\bS$. It turns out that $\bS$ is the Fra\"iss\'e limit of the class of finite sets equipped with a cyclic order and $\Sigma$-coloring (see Proposition~\ref{prop:S-homo}); in particular, up to isomorphism, $\bS$ is independent of the choice of coloring $\sigma$. It also follows that the automorphism group $G$ of $\bS$ is oligomorphic.

The goal of this paper is to classify the measures for $G$. It is easy to see that there is a universal measure valued in a certain ring $\Theta(G)$, and the problem of classifying measures for $G$ amounts to computing the ring $\Theta(G)$. This is what our main theorem accomplishes:

\begin{theorem} \label{mainthm}
Given a directed tree $T$ with edges labeled by $\Sigma$, there is an associated $\bZ$-valued measure $\mu_T$ for $G$. The product of these measures (over isomorphism classes of trees) defines a ring isomorphism $\Theta(G) \to \prod_T \bZ$. In particular, $\Theta(G) \cong \bZ^N$, where $N=2^n \cdot (n+1)^{n-2}$ and $n=\# \Sigma$.
\end{theorem}

Here, a ``tree'' is a connected simple graph with no cycles, ``directed'' means each edge has been given an orientation, and the labeling means that there is a  given bijection between the edge set and $\Sigma$. An example is given in \S \ref{ss:ex}.

Fix a point $\infty \in \bS$, and let $\bL=\bS \setminus \{\infty\}$. The set $\bL$ is totally ordered and $\Sigma$-colored, and its automorphism group $H$ is oligomorphic. We also classify measures for $H$ (see \S \ref{s:line}). (We note that the group $H$ appears in \cite[\S 5.2]{Agarwal} and \cite{Laflamme}.) 

\subsection{Summary of proof} \label{ss:plan}

The proof has essentially four main steps:
\begin{enumerate}
\item Using the classification of open subgroups of $H$ (Proposition~\ref{prop:subgp}), we first show that measures for $G$ are equivalent to $\Sigma$-measures. A $\Sigma$-measure is a rule $\nu$ that assigns to each $a,b \in \Sigma$ and word $w \in \Sigma^{\star}$ a quantity $\nu_{a,b}(w)$, such that certain relations hold (see Definition~\ref{defn:Sigma-meas}). This is an important reduction since $\Sigma$-measures are purely combinatorial objects.
\item We next show that $\Sigma$-measures are determined by their values on words of length one, and the defining relations are generated by those involving words of length at most two. We phrase this result as an equivalence between $\Sigma$-measures and another notion called $\Sigma$-symbols (see Definition~\ref{defn:symbol}). This is another important reduction since $\Sigma$-symbols are far simpler than $\Sigma$-measures.
\item Next, we essentially solve the defining equations for $\Sigma$-symbols. Assuming the coefficient ring is connected, $\Sigma$-symbols correspond to functions $S \colon \Sigma^2 \to \{0,1\}$ satisfying one relatively simple condition, namely, condition $(\ast)$ in \S \ref{ss:orbi}.
\item The functions $S$ appearing above are studied in \S \ref{s:bisect}, where they are called ``oriented bisection structures.'' We show that these functions naturally correspond to directed trees with edges labeled by $\Sigma$.
\end{enumerate}

\subsection{Tensor categories}

Let $k$ be a field. Given a tree $T$ as in Theorem~\ref{mainthm}, the machinery of \cite{repst} produces a $k$-linear rigid tensor category $\uPerm(G; \mu_T)$ (which is not abelian). This is potentially a very interesting example, especially in light of the results of \cite{delannoy, delannoy2} in the case $\# \Sigma=1$. An important problem is to determine if this category has an abelian envelope, as this would yield a new pre-Tannakian category. We know of no obstruction, but when $\# \Sigma>1$ the results in \cite{repst} on abelian envelopes do not apply (the measure $\mu_T$ is not quasi-regular).

\subsection{Outline}

In \S \ref{s:bisect}, we introduce the concept of a bisection structure, and show that they are equivalent to trees. In \S \ref{s:olig}, we review oligomorphic groups and measures in general. In \S \ref{s:group}, we introduce the main groups of interest, and determine some of their group-theoretic properties. In \S \ref{s:class}, we prove Theorem~\ref{mainthm}. Finally, in \S \ref{s:line}, we treat the group $H$.

\subsection{Notation}

The following is the most important notation:
\begin{description}[align=right,labelwidth=2.25cm,leftmargin=!]
\item[ $k$ ] the coefficient ring
\item[ $\Sigma$ ] the finite set of colors
\item[ $\Sigma^{\star}$ ] the set of words in the alphabet $\Sigma$
\item[ $\bS$ ] the circle
\item[ $\bL$ ] the line, defined as $\bS \setminus \{\infty\}$
\item[ $\sigma$ ] the coloring of $\bS$
\item[ $G$ ] the automorphism group of $\bS$ (except in \S \ref{s:olig})
\item[ $H$ ] the automorphism group of $\bL$ 
\item[ $G(A)$ ] the subgroup of $G$ fixing each element of $A$
\item[ {$w[i,j]$} ] a substring of the word $w$ (see \S \ref{ss:Sigma-meas})
\end{description}

\subsection*{Acknowledgments}

We thank Nate Harman and Steven Sam for helpful discussions.

\section{Bisection structures} \label{s:bisect}

\subsection{The definition} \label{ss:bisect-defn}

Suppose $x$ is a real number. Deleting $x$ from the real line cuts it into two pieces. We thus get an equivalence relation $R_x$ on $\bR \setminus \{x\}$ by letting $R_x(y,z)$ mean ``$y$ and $z$ belong to the same connected component.'' If $x$, $y$, and $z$ are distinct real numbers then exactly one of $R_x(y,z)$, $R_y(x,z)$, and $R_z(x,y)$ is false. The following definition axiomizes this situation, but allows for a bit more flexibility.

\begin{definition} \label{defn:bisect}
Let $X$ be a set. A \defn{bisection structure} on $X$ is a rule $R$ assigning to each $x \in X$ an equivalence relation $R_x$ on $X \setminus \{x\}$ such that the following conditions hold:
\begin{enumerate}
\item The equivalence relation $R_x$ has at most two equivalence classes.
\item If $x,y,z \in X$ are distinct then at most one of $R_x(y,z)$, $R_y(x, z)$, and $R_z(x, y)$ fails to hold. \qedhere
\end{enumerate}
\end{definition}

Trees also lead to bisection structures. For the purposes of this paper, a \defn{tree} is a finite simple graph that is connected and has no cycles, and an \defn{$X$-labeled tree} is a tree with a given bijection between $X$ and the edge set (i.e., the edges are labeled by $X$). Suppose $T$ is an $X$-labeled tree, and let $x$ be an edge of $T$. Deleting $x$ from $T$ (but not the vertices in $x$) yields a forest $T_x$ with at most two components. Let $R_x(y,z)$ mean ``$y$ and $z$ belong to the same component of $T_x$.'' Equivalently, $R_x(y,z)$ means that the geodesic joining $y$ and $z$ in $T$ does not include the edge $x$. One readily verifies that the $R_x$'s define a bisection structure on $X$. We denote this bisection structure by $R^T$.

\begin{remark}
Bisection structures are closely related to the concept of \emph{betweenness}. Indeed, if $R$ is the bisection structure on the real line then $\lnot R_x(y,z)$ exactly means that $x$ is between $y$ and $z$; a similar observation holds for the bisection structures associated to trees. Bankston \cite{Bankston} has defined a general notion of betweenness, and discussed many examples. The betweenness relations on the vertices of trees appears often in the literature; however, this does not lead to a bisection structure in general.
\end{remark}

\subsection{The main result}

The following is the main result we need on bisection structures:

\begin{theorem} \label{thm:bisect}
Let $R$ be a bisection structure on a finite set $X$. Then there exists an $X$-labeled tree $T$ such that $R=R^T$, and $T$ is unique up to isomorphism.
\end{theorem}

The proof will be divided into several lemmas. Say that distinct elements $x$ and $y$ of $X$ are \defn{adjacent} if for all $z \in X \setminus \{x,y\}$ we have $R_z(x,y)$. For each $x \in X$, let $E_x^+$ and $E_x^-$ be the two equivalence classes for $R_x$, labeled in an arbitrary manner; if there are fewer than two equivalence classes, take one or both of the sets to be empty. Let $\tilde{V} = X \times \{\pm\}$. We define a relation $\sim$ on $\tilde{V}$ by $(x,a) \sim (y,b)$ if $(x,a)=(y,b)$, or $x \ne y$ are adjacent and $x \in E_y^b$ and $y \in E_x^a$.

\begin{lemma}
$\sim$ is an equivalence relation on $\tilde{V}$.
\end{lemma}

\begin{proof}
We just need to verify transitivity. Thus suppose $(x,a) \sim (y,b)$ and $(y,b) \sim (z,c)$. We show $(x,a) \sim (z,c)$.

We first claim that $x$ and $z$ are adjacent. Let $w \in X \setminus \{x,z\}$ be given. If $w=y$ then since $x$ and $z$ both belong to the equivalence class $E_y^b$, the relation $R_w(x,z)$ holds. If $w \ne y$ then $R_w(x,y)$ and $R_w(y,z)$ hold, since $(x,y)$ and $(y,z)$ are adjacent pairs, and so $R_w(x,z)$ holds, since $R_w$ is an equivalence relation. This proves the claim.

Now, since $x$ and $y$ are adjacent, the relation $R_z(x,y)$ holds. Since $y$ belongs to $E_z^c$, it follows that $x$ does as well. Similarly, $z$ belongs to $E_x^a$. Thus $(x,a) \sim (z,c)$, as required.
\end{proof}

Let $V$ to be the quotient $\tilde{V}/\sim$. We define a graph $T=T(R)$ with vertex set $V$ and edge set $X$. For $x \in X$, the two vertices of the edge $x$ are the classes of $(x,+)$ and $(x,-)$ in $V$; these two vertices are distinct by definition of the equivalence relation, and so $T$ has no loops. The following lemma shows that there are no 2-cycles, i.e., parallel edges. Thus $T$ is a simple graph. Note that if $x$ and $y$ are distinct elements of $X$ then they are adjacent (in the above sense) if and only if they share a vertex in $T$.

\begin{lemma}
$T$ has no cycles, i.e., it is a forest.
\end{lemma}

\begin{proof}
Suppose by way of contraction that we have a cycle. Let $x_1, \ldots, x_n$ be the edges involved, so that $x_i$ is adjacent to $x_{i+1}$ for all $i \in \bZ/n$. Since the labeling of equivalence classes was arbitrary, we may as well suppose that $(x_i, +) \sim (x_{i+1}, -)$ for all $i \in \bZ/n$. We thus have $x_2 \in E^+_{x_1}$. Now, $x_1 \in E^-_{x_2}$ and $x_3 \in E^+_{x_2}$, so $\lnot R_{x_2}(x_1, x_3)$ holds; thus by Definition~\ref{defn:bisect}(b), we have $R_{x_1}(x_2,x_3)$, and so $x_3 \in E_{x_1}^+$. Continuing in this manner, we find $x_i \in E_{x_1}^+$ for all $2 \le i \le n$. However, since $(x_n, +) \sim (x_1, -)$, we have $x_n \in E_{x_1}^-$, a contradiction.
\end{proof}

Let $x \ne z$ be elements of $X$. We say that $y \in X \setminus \{x,z\}$ is \defn{between} $x$ and $z$ if $\lnot R_y(x,z)$ holds. We let $P(x,z)$ be the set of such elements $y$.

\begin{lemma}
$T$ is connected, and thus a tree. In fact, for $x \ne z$, the set $P(x,z)$ is the collection of edges in the shortest path joining $x$ and $z$.
\end{lemma}

\begin{proof}
We proceed by induction on $\# P(x,z)$. If $P(x,z)$ is empty then $x$ and $z$ are adjacent, and the statement is trivial. Suppose now that $\# P(x,z)>0$, and let $y$ be an element of this set. We claim that
\begin{displaymath}
P(x,z) = P(x,y) \cup \{y\} \cup P(y,z).
\end{displaymath}
Indeed, suppose that $w \in P(x,y)$, i.e., $\lnot R_w(x,y)$ holds. We have the following implications
\begin{displaymath}
\lnot R_w(x,y) \implies R_y(w,x) \implies \lnot R_y(w,z) \implies R_w(y,z) \implies \lnot R_w(x,z).
\end{displaymath}
The first implication comes from Definition~\ref{defn:bisect}(b); the second follows since $\lnot R_y(x,z)$ holds and $R_y$ is an equivalence relation; the third comes from Definition~\ref{defn:bisect}(b); and the fourth follows since $\lnot R_w(x,y)$ holds and $R_w$ is an equivalence relation. Thus $w \in (x,z)$. We have thus show that $P(x,y) \subset P(x,z)$, and by symmetry we have $P(y,z) \subset P(x,z)$ as well. This proves one of the containments above.

We now prove the reverse containment. Thus suppose $w \in P(x,z)$ and $w \ne y$. We must show $w$ belongs to either $P(x,y)$ or $P(y,z)$. Suppose it does not belong to $P(x,y)$. Then $\lnot R_w(x,z)$ holds, since $w \in P(x,z)$, and $R_w(x,y)$ holds, since $w \not\in P(x,y)$. Since $R_w$ is an equivalence relation, it follows that $\lnot R_w(y,z)$ holds, and so $w \in P(y,z)$, as required.

Now, $P(x,y)$ and $P(y,z)$ do not contain $y$, and thus are proper subsets of $P(x,y)$. Thus, by the inductive hypothesis, $P(x,y)$ is a path from $x$ to $y$, and $P(y,z)$ is a path from $y$ to $z$. It follows that $P(x,z)$ is a path from $x$ to $z$. As for the minimality of $P(x,z)$, choose $y \in P(x,z)$ adjacent to $x$ (which must exist). Then $P(x,z) = \{y\} \sqcup P(y,z)$. Since $P(y,z)$ is a minimal path between $y$ and $z$ (by induction), and $x \not\in P(y,z)$, it follows that $P(x,z)$ is a minimal path between $x$ and $z$.
\end{proof}

\begin{lemma} \label{lem:bisect-4}
We have $R=R^T$.
\end{lemma}

\begin{proof}
We have seen that $P(y,z)$ is the shortest path in $T$ from $y$ to $z$, and so $R^T_x(y,z)$ holds if and only if $x \not\in P(y,z)$. However, $x \not\in P(y,z)$ is equivalent to $R_x(y,z)$, by definition.
\end{proof}

\begin{proof}[Proof of Theorem~\ref{thm:bisect}]
Let $\sT$ be the set of isomorphism classes of $X$-labeled trees, and let $\sB$ be the set of bisection structures on $X$. We have a map $\Phi \colon \sT \to \sB$ by $\Phi(T)=R^T$. The construction $R \mapsto T(R)$ above yields a function $\Psi \colon \sB \to \sT$. Lemma~\ref{lem:bisect-4} shows that $\Phi \circ \Psi$ is the identity. It is easy to see directly that $\Psi \circ \Phi$ is also the identity; that is, if one starts with a tree $T$ then $T(R^T)$ is isomorphic to $T$.
\end{proof}

\subsection{Orientations} \label{ss:orbi}

We now discuss a variant of the above ideas. An \defn{oriented bisection structure} is a function
\begin{displaymath}
S \colon (X \times X) \setminus \Delta \to \{\pm \},
\end{displaymath}
where $\Delta$ denotes the diagonal, such that taking $R_x(y,z)$ to be relation defined by $S(x,y)=S(x,z)$, the collection $R=\{R_x\}$ is a bisection structure on $X$. This $R$ automatically satisfies Definition~\ref{defn:bisect}(a), so one only needs to consider Definition~\ref{defn:bisect}(b). In terms of $S$, this amounts to the following condition:
\begin{itemize}
\item[($\ast$)] Given distinct $x,y,z \in X$, at most one of the equalities
\begin{displaymath}
S(x,y)=S(x,z), \qquad S(y,x)=S(y,z), \qquad S(z,x)=S(z,y)
\end{displaymath}
fails to hold.
\end{itemize}
More informally, an oriented bisection structure is simply a bisection structure $R$ where for each $x$ we have labeled the equivalence classes of $R_x$ as $+$ and $-$. We made use of exactly this kind of structure in the proof of Theorem~\ref{thm:bisect}.

Define a \defn{directed tree} to be a tree in which each edge has been given a direction. If $T$ is an $X$-labeled directed tree then it induces an oriented bisection structure $S^T$ on $X$, as follows. If we delete edge $x$ from $T$, there are (at most) two resulting components; the edge $x$ points towards one of these components, and away from the other. We put $S^T(x,y)=+$ if $x$ points towards $y$'s component, and put $S^T(x,y)=-$ otherwise. One readily verifies that $S^T$ is an oriented bisection structure, and that the analog of Theorem~\ref{thm:bisect} holds in this setting. We give an example in \S \ref{ss:ex}.

\subsection{Enumeration}

The following proposition counts the structures we have considered.

\begin{proposition} \label{prop:enum}
Let $X$ be a finite set with $n$ elements. Put
\begin{align*}
N_1 &= \text{the number of bisection structures on $X$} \\
N_2 &= \text{the number of $X$-labeled trees up to isomorphism} \\
N_3 &= \text{the number of oriented bisection structures on $X$} \\
N_4 &= \text{the number of directed $X$-labeled trees up to isomorphism}
\end{align*}
Then
\begin{displaymath}
N_1=N_2=(n+1)^{n-2}, \qquad N_3=N_4=2^n (n+1)^{n-2}.
\end{displaymath}
The first formula is valid for $n \ge 2$, while the second is valid for $n \ge 1$.
\end{proposition}

\begin{proof}
The equality $N_1=N_2$ follows from Theorem~\ref{thm:bisect}. The explicit formula for $N_2$ follows from Cayley's theorem on trees; see \cite[Proposition~2.1]{CameronTrees}. The equality $N_3=N_4$ follows from the oriented analog of Theorem~\ref{thm:bisect}. If $n \ge 2$ then there are no automorphisms of an $X$-labeled tree $T$, since the group of unlabeled automorphisms acts faithfully on the edges (see the proof of \cite[Proposition~2.1]{CameronTrees}). It follows that if we direct the edges of $T$ in two different ways, the resulting directed $X$-labeled trees are non-isomorphic. Hence $N_4=2^n \cdot N_3$. The formula for $N_4$ when $n=1$ is easily verified directly. (Note that if $T$ is an $X$-labeled tree with one edge then the two ways of directing this edge yield isomorphic directed $X$-labeled trees.)
\end{proof}

\begin{remark}
The integer sequence defined by the formula $2^n \cdot (n+1)^{n-2}$ is discussed in \cite{OEIS}. We mention two other places where it occurs.
\begin{itemize}
\item Let $S_n=\bQ[x_i,y_i,z_i]_{1 \le i \le n}$, and let $R_n$ be the quotient of $S_n$ by the ideal generated by homogeneous $\fS_n$-invariants of positive degree. Here the symmetric group $\fS_n$ acts on $S_n$ by permuting each set of variables in the obvious manner. Haiman \cite[Fact~2.8.1]{Haiman} observed that the dimension of $R_n$ as a $\bQ$-vector space is $2^n (n+1)^{n-2}$ for $1 \le n \le 5$, and suggested this might be true for all $n$; as far as we know, this is still open.
\item Let $B_n(x)$ be the $n$th Morgan--Voyce polynomial. This is defined recursively by $B_0=B_1=1$ and $B_n=(x+2)B_{n-1} - B_{n-2}$. The disciminant of $B_{n+1}$ is $2^n(n+1)^{n-2}$ \cite[Table~5]{Florez}. \qedhere
\end{itemize}
\end{remark}

\section{Oligomorphic groups and measures} \label{s:olig}

\subsection{Oligomorphic groups}

An \defn{oligomorphic group} is a permutation group $(G, \Omega)$ such that $G$ has finitely many orbits on $\Omega^n$ for all $n \ge 1$. We refer to Cameron's book \cite{CameronBook} for general background on these groups.

Suppose we have an oligomorphic group $(G, \Omega)$. For a finite subset $A \subset \Omega$, let $G(A)$ be the subgroup of $G$ fixing each element of $A$. These subgroups form a neighborhood basis of the identity for a topology on $G$. This topology has three important properties: it is Hausdorff; it is non-archimedean (open subgroups form a neighborhood basis of the identity); and it is Roelcke pre-compact (if $U$ and $V$ are open subgroups then $U \backslash G/V$ is finite); see \cite[\S 2.2]{repst}. We say that a topological group is \defn{admissible} if it satisfies these three properties.

Although we ultimately care most about oligomorphic groups, our constructions only depend on the topology and not the specific permutation action, so we tend to work with admissible topological groups.

\subsection{Actions}

Let $G$ be an admissible topological group. We say that an action of $G$ on a set $X$ is \defn{smooth} if every point has open stabilizer. We use the term ``$G$-set'' to mean ``set equipped with a smooth action of $G$.'' We say that a $G$-set is \defn{finitary} if it has finitely many orbits. A product of two finitary $G$-sets is again a finitary $G$-set. See \cite[\S 2.3]{repst} for details.

A \defn{$\hat{G}$-set} is a $U$-set for some open subgroup $U$ of $G$, called a \defn{group of definition}; shrinking $U$ does not change the $\hat{G}$-set. A $\hat{G}$-set is called finitary if it is finitary with respect to some group of definition; this does not depend on the group of definition. If $f \colon X \to Y$ is a map of $G$-sets then the fiber over any point is a $\hat{G}$-set; this is one reason this concept is useful. The symbol $\hat{G}$ has no rigorous meaning on its own, but we think of it as an infinitesimal neighborhood of the identity. See \cite[\S 2.]{repst} for details.

\subsection{Measures}

Let $G$ be an admissible group. The following definition was introduced in \cite{repst}, and will be the primary concept studied in this paper:

\begin{definition} \label{defn:measure}
A \defn{measure} for $G$ valued in a commutative ring $k$ is a rule $\mu$ assigning to each finitary $\hat{G}$-set $X$ a quantity $\mu(X)$ in $k$ such that the following axioms hold (in which $X$ and $Y$ denote finitary $\hat{G}$-sets):
\begin{enumerate}
\item Isomorphism invariance: $\mu(X)=\mu(Y)$ if $X \cong Y$.
\item Normalization: $\mu(\bone)=1$, where $\bone$ is the one-point $\hat{G}$-set.
\item Conjugation invariance: $\mu(X^g)=\mu(X)$, where $X^g$ is the conjugate of the $\hat{G}$-set $X$ by $g \in G$.
\item Additivity: $\mu(X \amalg Y)=\mu(X)+\mu(Y)$.
\item Multiplicativity in fibrations: if $X \to Y$ is a map of transitive $U$-sets, for some open subgroup $U$, with fiber $F$ (over some point) then $\mu(X)=\mu(F) \cdot \mu(Y)$.
\end{enumerate}
We let $\cM_G(k)$ denote the set of $k$-valued measures for $G$.
\end{definition}

Given a measure, one obtains a theory of integration for functions on $G$-sets; see \cite[\S 3]{repst}. The construction of tensor categories in \cite{repst} is built on top of this theory of integration. For the present paper, however, we will not need integration.

There are a few concepts equivalent to the above notion of measure that we mention, simply to provide the reader with more intuition:
\begin{enumerate}
\item If $X \to Y$ is a map of $G$-sets then the fiber over any point is a $\hat{G}$-set, and every $\hat{G}$-set can be obtained in this manner. One can use this to reformulate the notion of measure as a rule that assigns to each such map $X \to Y$ (with $Y$ transitive and $X$ finitary) a quantity in $k$, such that certain conditions hold; see \cite[\S 4.5]{repst}. The advantage of this formulation is that it depends only on the category of $G$-sets.
\item A \defn{generalized index} is a rule assigning to each containment of open subgroups $U \subset V$ a quantity $\lbb U:V \rbb$ in $k$, satisfying properties similar to the usual index; see \cite[\S 3.6]{repst}. This concept is equivalent to measure, with $\mu$ corresponding to $\lbb -:- \rbb$ if $\lbb U:V \rbb=\mu(U/V)$ for all $V \subset U$.
\item Suppose that $\cC$ is a Fra\"iss\'e class with limit $\Omega$, and $G=\Aut(\Omega)$ is oligomorphic. A measure for $G$ is then equivalent to a rule assigning to each inclusion $X \subset Y$ in $\cC$ a value in $k$, such that certain conditions hold; see \cite[\S 6]{repst}. (Actually, this only gives a measure for $G$ relative to a stabilizer class.) This shows that measures are essentially combinatorial in nature.
\end{enumerate}
We will not use any of the above perspectives in this paper. However, we will give a combinatorial interpretation for our measures that is similar in spirit to (c).

There is one more concept connected to measures that we will use. Define a ring $\Theta(G)$ as follows: start with the polynomial ring in variables $[X]$, where $X$ varies over isomorphism classes of finitary $\hat{G}$-sets, and impose relations corresponding to Definition~\ref{defn:measure}(b,c,d,e). There is a measure $\mu_{\rm univ}$ valued in $\Theta(G)$ given by $\mu_{\rm univ}(X)=[X]$. This measure is universal, in the sense that if $\mu$ is a measure valued in some ring $k$ then there is a unique ring homomorphism $\phi \colon \Theta(G) \to k$ such that $\mu(X)=\phi(\mu_{\rm univ}(X))$. A complete understanding of measures for $G$ essentially amounts to computing the ring $\Theta(G)$.

\begin{example}
Let $\fS$ be the infinite symmetric group, and let $\Omega=\{1,2,3,\ldots\}$ be its domain. Given a complex number $t$, there is a unique $\bC$-valued measure $\mu_t$ for $\fS$ such that $\mu_t(\Omega)=t$. This measure satisfies $\mu_t(\Omega^{(n)})=\binom{t}{n}$, where $\Omega^{(n)}$ denotes the set of $n$-element subsets of $\Omega$. The ring $\Theta(\fS)$ is the ring of integer-valued polynomials in a single variable. These statements are proven in \cite[\S 15]{repst}.
\end{example}

\subsection{Minimal maps}

Let $G$ be an admissible group and let $\phi \colon X \to Y$ be a map of transitive $G$-sets. We say that $\phi$ is \defn{minimal} if it is not an isomorphism and does not factor non-trivially, i.e., given $\phi=\beta \circ \alpha$, where $\beta$ and $\alpha$ are maps of transitive $G$-sets, either $\alpha$ or $\beta$ is an isomorphism. The following two results show the significance of this notion.

\begin{proposition} \label{prop:min-subgp}
Given an open subgroup $V$ of $G$, there are only finitely many subgroups $U$ of $G$ containing $V$.
\end{proposition}

\begin{proof}
If $U$ contains $V$ then $U$ is a union of double cosets of $V$. Since $V \backslash G/V$ is finite, there are thus only finitely many possibilities for $U$.
\end{proof}

\begin{proposition}
Any map of transitive $G$-sets that is not an isomorphism can be factored into a sequence of minimal maps.
\end{proposition}

\begin{proof}
It suffices to show that if $V \subset U$ is a proper inclusion of open subgroups then the natural map $G/V \to G/U$ admits such a factorization. Choose a strict chain of open subgroups
\begin{displaymath}
V=W_0 \subset W_1 \subset \cdots \subset W_n=U
\end{displaymath}
where $W_i$ is minimal over $W_{i-1}$; such a chain exists by Proposition~\ref{prop:min-subgp}. The map $G/V \to G/U$ factors as the composition of the minimal maps $G/W_{i-1} \to G/W_i$, which completes the proof.
\end{proof}

\subsection{Another view on measures} \label{ss:E-meas}

We now give a slight reformulation of the definition of measure that will be more convenient in our particular case. For an admissible group $G$, let $\Omega(G)$ be its \defn{Burnside ring}. This is the free $\bZ$-module on the set of isomorphism classes of transitive $G$-sets. For a transitive $G$-set $X$, we let $\lbb X \rbb$ denote its class in $\Omega(G)$. For a general finitary $G$-set $X$, we define $\lbb X \rbb=\sum_{i=1}^n \lbb Y_i \rbb$, where $Y_1, \ldots, Y_n$ are the $G$-orbits on $X$. As the name suggests, $\Omega(G)$ is a ring, via $\lbb X \rbb \cdot \lbb Y \rbb=\lbb X \times Y \rbb$.

Let $\sE$ be a collection of open subgroups of $G$ that is stable under conjugation, and such that every open subgroup contains some member of $\sE$. We introduce the following notion:

\begin{definition} \label{defn:E-meas}
An \defn{$\sE$-measure} valued in a ring $k$ is a rule $\mu_{\bullet}$ assigning to each $U \in \sE$ an additive map $\mu_U \colon \Omega(U) \to k$ satisfying the following axioms:
\begin{enumerate}
\item We have $\mu_U(\bone)=1$ for any $U \in \sE$.
\item Given subgroups $V \subset U$ in $\sE$ and a finitary $U$-set $X$, we have $\mu_U(X)=\mu_V(X)$.
\item Given $U \in \sE$, a finitary $U$-set $X$ and $g \in G$, we have $\mu_U(X)=\mu_{U^g}(X^g)$, where the superscript denotes conjugation.
\item Given $U \in \sE$ and a map $\pi \colon X \to Y$ of transitive $U$-sets, we have $\mu_U(X)=\mu_U(Y) \mu_V(F)$, where $F=\pi^{-1}(y)$ for some $y \in Y$, and $V \in \sE$ stabilizes $y$.
\end{enumerate}
We have written $\mu_U(X)$ in place of $\mu_U(\lbb X \rbb)$ above. Let $\cM^{\sE}_G(k)$ denote the set of $\sE$-measures for $G$ valued in $k$.
\end{definition}

The following is the main result we require on this concept.

\begin{proposition} \label{prop:meas-E-meas}
We have a natural isomorphism $\cM_G(k) \to \cM^{\sE}_G(k)$.
\end{proposition}

\begin{proof}
Suppose that $\mu$ is a measure for $G$. For an open subgroup $U$, define $\mu_U \colon \Omega(U) \to k$ by $\mu_U(X)=\mu(X)$. It is clear that $\mu_{\bullet}$ is an $\sE$-measure, and that $\mu$ can be recovered from $\mu_{\bullet}$. We thus have an injective map $\Phi \colon \cM_G(k) \to \cM^{\sE}_G(k)$.

We now show that $\Phi$ is surjective, which will complete the proof. Let $\mu_{\bullet}$ be a given $\sE$-measure. We define a measure $\mu$, as follows. Let $X$ be a finitary $\hat{G}$-set. Choose a group of definition $U$ for $X$ that belongs to $\sE$, and put $\mu(X)=\mu_U(X)$. This is independent of the choice of $U$. Indeed, suppose $U'$ is a second group of definition belonging to $\sE$. Then $U \cap U'$ is an open subgroup of $G$, and thus contains some $V \in \sE$ by hypothesis. We then have $\mu_U(X)=\mu_V(X)=\mu_{U'}(X)$, since $\mu_{\bullet}$ is compatible with restriction. Clearly, $\mu_{\bullet}=\Phi(\mu)$, provided that $\mu$ is a measure, so it suffices to show this. One easily sees that that $\mu$ satisfies axioms (a)--(d) of Definition~\ref{defn:measure}.

We now verify axiom (e). Thus let $\pi \colon X \to Y$ be a map of transitive $U$-sets, for some open subgroup $U$, and let $F=\pi^{-1}(y)$ for some point $y \in Y$. Let $V \subset U$ be an open subgroup contained in $\sE$. Let $Y=\bigsqcup_{i=1}^n Y_i$ be the decomposition of $Y$ into orbits of $V$ and let $X_i=\pi^{-1}(Y_i)$. Let $y_i \in Y_i$ be any point, and let $F_i=\pi^{-1}(y_i)$ be the fiber over it. Let $W \subset V$ be an open subgroup in $\sE$ fixing each $y_i$. We have
\begin{displaymath}
\mu(X_i)=\mu_V(X_i)=\mu_V(Y_i) \mu_W(F_i) = \mu(Y_i) \mu(F_i) = \mu(Y_i) \mu(F).
\end{displaymath}
In the first step, we used the definition of $\mu$; in the second step, we used Definition~\ref{defn:E-meas}(d); in the third step, we again used the definition of $\mu$; and in the final step, we used that $F_i$ is conjugate to $F$ by $U$ (since $U$ acts transitively on $Y$), and that $\mu$ is conjugation invariant. Summing the above equation over $i$, we find $\mu(X)=\mu(Y) \mu(F)$, as required.
\end{proof}

We note that since measures are multiplicative, i.e., $\mu(X \times Y)=\mu(X) \cdot \mu(Y)$, the above proposition shows that if $\mu_{\bullet}$ is an $\sE$-measure then each map $\mu_U \colon \Omega(U) \to k$ is a ring homomorphism. Our next result can simplify the task of verifying Definition~\ref{defn:E-meas}(d).

\begin{proposition} \label{prop:minimal}
Suppose that $\sE$ is downwards closed, meaning that if $U \in \sE$ and $V \subset U$ is an open subgroup then $V \in \sE$. For each $U \in \sE$, let $\mu_U \colon \Omega(U) \to k$ be an additive map. Suppose $\mu_{\bullet}$ satisfies Definition~\ref{defn:E-meas}(a,b,c) as well as the following condition:
\begin{itemize}
\item[(d')] Given $U \in \sE$ and a minimal map $\pi \colon X \to Y$ of transitive $U$-sets, we have $\mu_U(X)=\mu_U(Y) \mu_V(F)$, where $F=\pi^{-1}(y)$ for some $y \in Y$, and $V \in \sE$ stabilizes $y$.
\end{itemize}
Then $\mu_{\bullet}$ also satisfies Definition~\ref{defn:E-meas}(d), and is thus an $\sE$-measure.
\end{proposition}

\begin{proof}
Given open subgroups $V \subset U$ of $G$, define $\delta(V \subset U)$ to be the maximal $n$ for which there exists a chain $V=W_0 \subsetneq \cdots \subsetneq W_n=U$ of subgroups. This is defined by Proposition~\ref{prop:min-subgp}. We note that $\delta(V \subset U)=0$ if and only if $V=U$, and $\delta(V \subset U)=1$ if and only if $V \ne U$ but there is no subgroup strictly between $V$ and $U$.

To prove Definition~\ref{defn:E-meas}(d), it suffices to show
\begin{displaymath}
\mu_U(U/W)=\mu_U(U/V) \mu_V(V/W)
\end{displaymath}
whenever $W \subset V \subset U$ are subgroups in $\sE$. We proceed by induction on $\delta(W \subset V)$. If $\delta(W \subset V)=0$ the statement is clear. Suppose now that $\delta(W \subset V)$ is positive, and let $W'$ be a minimal subgroup over $W$ contained in $V$, so that $\delta(W' \subset V)<\delta(W \subset V)$. Note that $U/W \to U/W'$ is a minimal map of transitive $U$-sets, and $V/W \to V/W'$ is a minimal map of transitive $V$-sets. We have
\begin{align*}
\mu_U(U/W)
&= \mu_U(U/W') \mu_{W'}(W'/W) \\
&= \mu_U(U/V) \mu_V(V/W') \mu_{W'}(W'/W) \\
&= \mu_U(U/V) \mu_V(V/W)
\end{align*}
where in the first step we used (d'), in the second step the inductive hypothesis, and in the third (d') again. The result follows.
\end{proof}

\section{The colored circle and its symmetries} \label{s:group}

\subsection{The circle}

Fix a countable set $\bS$ with an everywhere dense cyclic order; for example, one can take the roots of unity in the complex unit circle. Let $\Sigma$ be a non-empty finite set and let $\sigma \colon \bS \to \Sigma$ be a function such that $\sigma^{-1}(a)$ is dense for every $a \in \Sigma$. We regard $\sigma$ as a coloring of $\bS$. Let $G$ be the automorphism group of $\bS$, i.e., the group of all self-bijections preserving the cyclic ordering and the coloring.

\begin{proposition} \label{prop:S-homo}
We have the following:
\begin{enumerate}
\item $\bS$ is a homogeneous structure: if $X$ and $Y$ are finite subsets of $\bS$ and $i \colon X \to Y$ is an isomorphism (i.e., a bijection preserving the induced cyclic orders and colorings) then there exists $g \in G$ such that $i = g \vert_X$.
\item $\bS$ is the Fra\"iss\'e limit of the class of finite sets equipped with a cyclic order and $\Sigma$-coloring; in particular, $\bS$ is independent of the choice of $\sigma$, up to isomorphism.
\item The group $G$ is oligomorphic (with respect to its action on $\bS$).
\end{enumerate}
\end{proposition}

\begin{proof}
(a) For this proof, a ``structure'' means a set equipped with a cyclic order and a $\Sigma$-coloring. Suppose that $X \to Y$ is an embedding of finite structures and we have an embedding $\alpha \colon X \to \bS$. We claim that $\alpha$ extends to an embedding $\beta \colon Y \to \bS$. By an inductive argument, it suffices to treat the case where $Y$ has one more element than $X$. Thus suppose $Y=X \sqcup \{y\}$. If $\# X \le 1$, the claim is clear, so suppose $\# X \ge 2$. Write $X=\{x_1, \ldots, x_n\}$, where $x_1<x_2<\cdots<x_n<x_1$, and let $i$ be such that $x_i<y<x_{i+1}$ (where $i+1$ is taken modulo $n$). Choose a point $z \in \bS$ between $\alpha(x_i)$ and $\alpha(x_{i+1})$ having the same color at $y$. Now simply define $\beta(x_j)=\alpha(x_j)$ for all $j$, and $\beta(y)=z$. Then $\beta$ is an embedding of $Y$ extending $\alpha$.

It now follows from a standard back-and-forth argument that $\bS$ is homogeneous. To be a bit more precise, the previous paragraph shows that $\bS$ is ``f-injective'' in the terminology of \cite[\S A.4]{homoten}. By \cite[Proposition~A.7]{homoten}, any f-injective object is homogeneous.

(b) Since $\bS$ is a countable homogeneous structure into which every finite structure embeds, it is the Fra\"iss\'e limit of the class of finite structures. (To see that every structure embeds, simply note that the empty structure does and so the general case follows from f-injectivity.) The Fra\"iss\'e limit is unique up to isomorphism, which yields the uniqueness statement.

(c) Let $\bS^{(n)}$ denote the set of $n$-element subsets of $\bS$. If $x$ and $y$ are two points in $\bS^{(n)}$ that are isomorphic (with their induced structures) then the homogeneity of $\bS$ shows that they belong to the same $G$-orbit. Since there are finitely many structures of cardinality $n$, it follows that $G$ has finitely many orbits on $\bS^{(n)}$. Since this holds for all $n$, it follows that $G$ is oligomorphic.
\end{proof}

\subsection{The line}

Fix a point $\infty \in \bS$, and let $\bL = \bS \setminus \{\infty\}$. Then $\bL$ carries a total order and a $\Sigma$-coloring. As a totally ordered set, $\bL$ is isomorphic to the set of rational numbers with its standard order. An argument similar to the above shows that $\bL$ is homogeneous, and the Fra\"iss\'e limit of the class of finite sets equipped with a total order and $\Sigma$-coloring. We let $H$ be the automorphism group of $\bL$, which is also oligomorphic. Note that $H$ is simply the stabilizer of $\infty$ in $G$, and for a finite subset $A$ of $\bL$, we have $H(A)=G(A \cup \{\infty\})$. (Here $H(A)$ denotes the subgroup of $H$ fixing each element of $A$.)

\subsection{Intervals}

Given $x \ne y$ in $\bS$, consider the set $I=(x,y)$ of all points $z \in \bS$ satisfying $x<z<y$. We refer to sets of the form $(x,y)$ as \defn{proper intervals}. We refer to $x$ and $y$ as the left and right endpoints of $I$. If $J=(x',y')$ is a second proper interval then $J=gI$ for some $g \in G$ if and only if $\sigma(x)=\sigma(x')$ and $\sigma(y)=\sigma(y')$. An \defn{improper interval} is one of the form $\bS \setminus \{x\}$ for $x \in \bS$.

Let $I$ be an interval (proper or improper). Then $I$ carries a total order and a $\Sigma$-coloring, and is easily seen to be the Fra\"iss\'e limit of the class of finite sets with a total order and $\Sigma$-coloring. Thus $I$ is abstractly isomorphic to $\bL$. We let $H_I$ be the automorphism group of $I$. The pair $(H_I, I)$ is isomorphic to $(H, \bL)$; in particular, $H_I$ is oligomorphic.

Suppose that $A$ is a finite non-empty subset of $\bS$, and write $\bS \setminus A = I_1 \sqcup \cdots \sqcup I_r$, where $I_1, \ldots, I_r$ are intervals. Whenever we write such a decomposition, we assume that the indexing of the intervals is compatible with the natural cyclic order on them, i.e., $I_i$ is between $I_{i-1}$ and $I_{i+1}$ (where the indices are taken modulo $n$). An element of $G(A)$ preserves each $I_i$, and so there is a natural map
\begin{displaymath}
G(A) \to H_{I_1} \times \cdots H_{I_r}.
\end{displaymath}
One readily sees that this map is an isomorphism.

\subsection{Open subgroups}

We now classify the open subgroups of $H$.

\begin{proposition} \label{prop:subgp}
Every open subgroup of $H$ has the form $H(A)$ for some finite subset $A \subset \bL$.
\end{proposition}

\begin{proof}
This is proved for $\# \Sigma=1$ in \cite[Proposition~17.1]{repst}. The general case follows from a similar argument.
\end{proof}

\begin{corollary} \label{cor:subgp}
Let $A$ be a non-empty finite subset of $\bS$. Then every open subgroup of $G(A)$ has the form $G(B)$ for some finite subset $B$ of $\bS$ containing $A$.
\end{corollary}

\begin{proof}
Since the choice of $\infty$ is arbitrary, we may as well assume it belongs to $A$. Thus $G(A) \subset H$, and the result now follows from the proposition.
\end{proof}

\begin{remark}
The $G(A)$'s do not account for all open subgroups of $G$. Let $G[A]$ denote the subgroup of $G$ that maps $A$ to itself (as a set). This can be larger than $G(A)$; for instance, if every point in $A$ has the same color, then the points of $A$ can be cyclically permuted, and $G[A] = \bZ/n \ltimes G(A)$ where $n=\# A$. One can show that $G[A]$ is the normalizer of $G(A)$, and that $G[A]/G(A)$ is a finite cyclic group. One can furthermore show that every open subgroup of $G$ sits between $G(A)$ and $G[A]$ for some $A$. We will not need this result, however.
\end{remark}

\subsection{Actions}

Let $\Sigma^{\star}$ denote the set of all words in the alphabet $\Sigma$. Given a word $w=w_1 \cdots w_n$ in $\Sigma^{\star}$ and an interval $I$, we let $I^w$ denote the subset of $I^n$ consisting of those tuples $(x_1, \ldots, x_n)$ such that $x_1<\cdots<x_n$ and $\sigma(x_i)=w_i$. The group $H_I$ clearly acts on $I^w$, and this action is transitive by the homogeneity of $I$.

\begin{proposition} \label{prop:Ghat-sets}
Let $A$ be a non-empty finite subset of $\bS$, and write $\bS \setminus A = I_1 \sqcup \cdots \sqcup I_r$. Then every transitive $G(A)$-set is isomorphic to $I_1^{w_1} \times \cdots \times I_r^{w_r}$ for some $w_1, \ldots, w_r \in \Sigma^{\star}$.
\end{proposition}

\begin{proof}
Let $X$ be a transitive $G(A)$-set. Then $X$ is isomorphic to $G(A)/U$ for some open subgroup $U$ of $G(A)$. By Corollary~\ref{cor:subgp}, we have $U=G(B)$ for some finite subset $B$ of $\bS$ containing $A$. Let $B_i = B \cap I_i$. Writing $B_i=\{x_{i,1}<\cdots<x_{i,n(i)}\}$, let $w_i=\sigma(x_{i,1}) \cdots \sigma(x_{i,n(i)})$. Under the isomorphism $G(A) = H_{I_1} \times \cdots \times H_{I_r}$, we have $G(B) = H_{I_1}(B_1) \times \cdots \times H_{I_r}(B_r)$. As $H_{I_i}$ acts transitively on $I_i^{w_i}$ with stabilizer $H_{I_i}(B_i)$, it follows that $H_{I_i}/H_{I_i}(B_i) \cong I_i^{w_i}$. Thus $X \cong I_1^{w_1} \times \cdots \times I_r^{w_r}$, as required.
\end{proof}

\begin{proposition} \label{prop:minimal-class}
Let $A$ be a non-empty finite subset of $\bS$, and write $\bS \setminus A = I_1 \sqcup \cdots \sqcup I_r$.
\begin{enumerate}
\item Let $w,w',w_2, \ldots, w_r \in \Sigma^{\star}$ and let $c \in \Sigma$. Let $\pi_1 \colon I_1^{wcw'} \to I_1^{ww'}$ be the map that projects away from the $c$ coordinate, and let
\begin{displaymath}
\pi \colon I_1^{wcw'} \times I_2^{w_2} \times \cdots \times I_r^{w_r} \to I_1^{ww'} \times I_2^{w_2} \times \cdots \times I_r^{w_r}
\end{displaymath}
be the map that is $\pi_1$ on the first factor and the identity on other factors. Then $\pi$ is a minimal map of transitive $G(A)$-sets.
\item The fiber of $\pi$ over any point is isomorphic to $J^c$, where $J$ is a subinterval of $I_1$ defined as follows. If $w$ is non-empty, let $y$ be a point in $I_1$ whose color is the final letter of $w$; otherwise, let $y$ be the left endpoint of $I_1$. If $w'$ is non-empty, let $y<y'$ be a point in $I_1$ whose color is the first letter of $w'$; otherwise, let $y'$ be the right endpoint of $I_1$. Then $J=(y,y')$.
\item Any minimal map of transitive $G(A)$-sets is isomorphic to one as in (a), after possibly re-indexed the intervals.
\end{enumerate}
\end{proposition}

\begin{proof}
Let $A \subset B \subset C$ be finite subsets of $\bS$ with $\# C=\# B + 1$. Then $G(B)$ is a minimal subgroup over $G(C)$, and so the natural map $G(A)/G(C) \to G(A)/G(B)$ is a minimal map of transitive $G(A)$-sets. Looking at the identifications in the proof of Proposition~\ref{prop:Ghat-sets}, we see that this map has the form stated in (a). Every minimal map has this form by the classification of open subgroups of $G(A)$.

We now explain statement (b). First note that the fiber of $\pi$ is isomorphic to the fiber of $\pi_1$, so we just consider this. Suppose $\ell(w)=n$ and $\ell(w')=m$. Let $p=(x_1, \ldots, x_n, z_1, \ldots, z_m)$ be a point in $I_1^{ww'}$. Let $J$ be the interval $(x_n, z_1)$, where $x_n$ is taken to be the left endpoint of $I_1$ if $w$ is empty, and $z_1$ is taken to be the right endpoint of $I_1$ if $w'$ is empty. Then $\pi_1^{-1}(p)=J^c$, and so (b) follows.
\end{proof}

\section{Classification of measures} \label{s:class}

\subsection{Combinatorial reformulation of measures} \label{ss:Sigma-meas}

In \S \ref{s:class}, we prove Theorem~\ref{mainthm} following the plan in \S \ref{ss:plan}. As a first step, we introduce $\Sigma$-measures and connect them to measures. For a word $w=w_1 \cdots w_n$ in $\Sigma^{\star}$, we let $w[i,j]$ denote the subword $w_i \cdots w_j$. We use parentheses to omit endpoints, e.g., $w[i,j)=w_i \cdots w_{j-1}$.

\begin{definition} \label{defn:Sigma-meas}
A \defn{$\Sigma$-measure} with values in $k$ is a rule $\nu$ assigning to each $a,b \in \Sigma$ and $w \in \Sigma^{\star}$ a quantity $\nu_{a,b}(w)$ in $k$ such that the following axioms hold:
\begin{enumerate}
\item $\nu_{a,b}(w)=1$ if $w$ is the empty word.
\item Let $w$ and $w'$ be words, put $r=\ell(w)$, and let $a,b,c \in \Sigma$. Then
\begin{displaymath}
\nu_{a,b}(w c w')=\nu_{a,b}(ww') \nu_{w_r,w'_1}(c).
\end{displaymath}
Here, we use the convention that $w_r=a$ if $w$ is empty, and $w'_1=b$ if $w'$ is empty.
\item Let $w$ of length $n$, and let $a,b,c \in \Sigma$. Then
\begin{displaymath}
\nu_{a,b}(w) = \sum_{i=0}^n \nu_{a,c}(w[1,i]) \nu_{c,b}(w(i,n]) + \sum_{w_i=c} \nu_{a,c}(w[1,i)) \nu_{c,b}(w(i,n]).
\end{displaymath}
\end{enumerate}
We let $\cM_{\Sigma}(k)$ denote the set of such measures.
\end{definition}

\begin{proposition} \label{prop:Sigma-meas}
We have a natural bijection $\cM_G(k) \cong \cM_{\Sigma}(k)$. Under this bijection, a measure $\mu$ for $G$ corresponds to a $\Sigma$-measure $\nu$ if and only if $\mu(I^w)=\nu_{a,b}(w)$ whenever $I$ is a proper interval with endpoints of color $a$ and $b$, and $w \in \Sigma^{\star}$.
\end{proposition}

The proof of the proposition will take the remainder of \S \ref{ss:Sigma-meas}. Define $R$ to be the commutative ring generated by symbols $x_{a,b}(w)$, where $a,b \in \Sigma$ and $w \in \Sigma^{\star}$, modulo the following relations:
\begin{enumerate}
\item $x_{a,b}(w)=1$ if $w$ is empty.
\item Let $w$ and $w'$ be words, put $r=\ell(w)$, and let $a,b,c \in \Sigma$. Then
\begin{displaymath}
x_{a,b}(w c w')=x_{a,b}(ww') x_{w_r,w'_1}(c),
\end{displaymath}
where we use conventions as in Definition~\ref{defn:Sigma-meas}.
\item Let $w$ of length $n$, let $a,b,c \in \Sigma$. Then
\begin{displaymath}
x_{a,b}(w) = \sum_{i=0}^n x_{a,c}(w[1,i]) x_{c,b}(w(i,n]) + \sum_{w_i=c} x_{a,c}(w[1,i)) x_{c,b}(w(i,n]).
\end{displaymath}
\end{enumerate}
Thus a $\Sigma$-measure is a homomorphism $R \to k$. With this language, we can reformulate Proposition~\ref{prop:Sigma-meas} as follows:

\begin{proposition} \label{prop:Sigma-meas-2}
There exists a ring isomorphism $\phi \colon \Theta(G) \to R$ satisfying $\phi([I^w])=x_{a,b}(w)$, where $I$ is an arbitrary proper interval, and $a$ and $b$ are the colors of the left and right endpoints of $I$.
\end{proposition}

We note that the classes $[I^w]$ generate $\Theta(G)$ by Proposition~\ref{prop:Ghat-sets}, so there is at most one ring isomorphism as in the proposition.

Let $\sE$ be the set of subgroups of $G$ of the form $G(A)$ where $\# A \ge 2$. This satisfies the conditions of \S \ref{ss:E-meas}. We construct an $R$-valued $\sE$-measure $\phi_{\bullet}$ for $G$. Let $A$ be a finite subset of $\bS$ of cardinality at least~2, and write $\bS \setminus A = I_1 \sqcup \cdots \sqcup I_r$. Recall that every transitive $G(A)$ set is isomorphic to one of the form $I_1^{w_1} \times \cdots \times I_r^{w_r}$ with $w_1, \ldots, w_r \in \Sigma^{\star}$ (Proposition~\ref{prop:Ghat-sets}). We define
\begin{displaymath}
\phi_{G(A)} \colon \Omega(G(A)) \to R
\end{displaymath}
to be the unique additive map satisyfing
\begin{displaymath}
\phi(I_1^{w_1} \times \cdots \times I_r^{w_r}) = x_{a_1,b_1}(w_1) \cdots x_{a_r,b_r}(w_r),
\end{displaymath}
where $a_i$ and $b_i$ are the colors of the left and right endpoints of $I_i$. We now verify that the system $\phi_{\bullet}$ is indeed an $\sE$-measure. Conditions (a) and (c) of Definition~\ref{defn:E-meas} are clear.

\begin{lemma} \label{lem:Sigma-meas-1}
Let $I$ be an interval and let $z \in I$. Write $I=J \sqcup \{z\} \sqcup K$ for intervals $J$ and $K$, and let $c=\sigma(z)$. Then for a word $w \in \Sigma^{\star}$ of length $n$, we have a natural bijection
\begin{displaymath}
I^w = \big( \coprod_{i=0}^n J^{w[1,i]} \times K^{w(i,n]} \big) \amalg \big( \coprod_{w_i=c} J^{w[1,i)} \times K^{w(i,n]} \big)
\end{displaymath}
that is equivariant for the action of $H_I(b) = H_J \times H_K$.
\end{lemma}

\begin{proof}
Recall that $I^w$ consists of tuples $x=(x_1<\cdots<x_n)$ in $I^n$ such that $\sigma(x_i)=w_i$. Let $X$ be the subset of $I^w$ consisting of points $x$ such that no $x_i$ is equal to $z$, and let $Y$ be the complement. We have a decomposition $X=X_0 \sqcup \cdots \sqcup X_n$, where $X_i$ is the subset of $X$ consisting of points $x$ such that $x_i<z<x_{i+1}$ (and where we ignore conditions involving $x_0$ or $x_{n+1}$), and an isomorphism $X_i \cong J^{w[1,i]} \times K^{w(i,n]}$. We also have a decomposition $Y=\bigsqcup_{w_i=c} Y_i$, where $Y_i$ consists of points $x$ such that $x_i=z$, and an isomorphism $Y_i \cong J^{w[1,i)} \times K^{w(i,n]}$. This completes the proof.
\end{proof}

\begin{lemma}
The system $\phi_{\bullet}$ satisfies Definition~\ref{defn:E-meas}(b).
\end{lemma}

\begin{proof}
It suffices to treat the case where $V$ is a maximal subgroup of $U$. We can thus assume $V=G(B)$ and $U=G(A)$ where $B=A \cup \{z\}$ and $z$ is some element of $\bS \setminus A$; let $c=\sigma(z)$ be the color of $z$. Write $\bS \setminus A = I_1 \sqcup \cdots \sqcup I_r$ as above; cyclically rotating the labels, if necessary, we assume that $z \in I_1$. Write $I_1=J \sqcup \{z\} \sqcup K$. Let $X=I_1^{w_1} \times \cdots \times I_r^{w_r}$ be a transitive $G(A)$-set. Decomposing $I_1^{w_1}$ by Lemma~\ref{lem:Sigma-meas-1}, we find that $\phi_V(X)$ is equal to
\begin{displaymath}
\big( \sum_{i=0}^n x_{a_1,c}(w[1,i]) x_{c,b_1}(w(i,n]) + \sum_{w_i=c} x_{a_1,c}(w[1,i)) x_{c,b_1}(w(i,n]) \big) \times x_{a_2,b_2}(w_2) \cdots x_{a_r,b_r}(w_r)
\end{displaymath}
By definition of $R$, the first factor is equal to $x_{a_1,b_1}(w_1)$, and so the whole expression is equal to $\phi_U(X)$. Thus $\phi_V(X)=\phi_U(X)$, as required.
\end{proof}

\begin{lemma}
The system $\phi_{\bullet}$ satisfies Definition~\ref{defn:E-meas}(d).
\end{lemma}

\begin{proof}
It follows from the classification of open subgroups of $H$ (Proposition~\ref{prop:subgp}) that $\sE$ satisfies the condition of Proposition~\ref{prop:minimal}. Thus, by that proposition, it is enough to verify Proposition~\ref{prop:minimal}(d'). Let $A$ be a finite subset of $\bS$ of cardinality at least~2, and write $\bS \setminus A = I_1 \sqcup \cdots \sqcup I_r$. Let $a_i$ and $b_i$ be the colors of the endpoints of $I_i$. Let $\pi \colon X \to Y$ be a minimal map of transitive $G(A)$-sets with fiber $F$. By Proposition~\ref{prop:minimal-class}, after possibly reindexing, $\pi$ is isomorphic to
\begin{displaymath}
\pi_1 \times \id \times \cdots \times \id \colon I_1^{wcw'} \times I_2^{w_2} \times \cdots \times I_r^{w_r} \to I_1^{ww'} \times I_2^{w_2} \times \cdots \times I_r^{w_r},
\end{displaymath}
where $\pi_1$ projects away from the $c$ coordinate. Additionally, $F$ is isomorphic to $J^c$, where $J$ is an interval with endpoints of colors $w_r$ and $w'_1$, with $r=\ell(w)$; here we use the convention that $w_r=a_1$ if $w=\emptyset$, and $w'_1=b_1$ if $w'=\emptyset$. The equation $\phi(X)=\phi(Y) \phi(F)$ thus becomes
\begin{displaymath}
x_{a_1,b_1}(wcw') x_{a_2,b_2}(w_2) \cdots x_{a_r,b_r}(w_r) = x_{a_1,b_1}(ww') x_{w_r,w_1'}(c) x_{a_2,b_2}(w_2) \cdots x_{a_r,b_r}(w_r),
\end{displaymath}
which does indeed hold in $R$: this is just the definining relation (b) of $R$, multiplied on each side by the same quantity.
\end{proof}

We have thus verified that the system $\phi_{\bullet}$ is an $\sE$-measure. By Proposition~\ref{prop:meas-E-meas}, $\phi_{\bullet}$ corresponds to a measure for $G$ valued in $R$, i.e., a ring homomorphism $\phi \colon \Theta(G) \to R$. This homomorphism clearly satisfies $\phi([I^w])=x_{a,b}(w)$, where $a$ and $b$ are the colors of the endpoints of $I$. The following lemma completes the proof of the proposition.

\begin{lemma}
The map $\phi \colon \Theta(G) \to R$ is an isomorphism.
\end{lemma}

\begin{proof}
Let $\tilde{R}$ be the polynomial ring in the symbols $x_{a,b}(w)$. Define a ring homomorphism $\tilde{\psi} \colon \tilde{R} \to \Theta(G)$ by $\psi(x_{a,b}(w))=[I^w]$, where $I$ is any proper interval with endpoints of colors $a$ and $b$. This is well-defined since if $J$ is a second such interval then $I$ and $J$ are conjugate by an element of $G$, and so $[I^w]=[J^w]$ in $\Theta(G)$. By computations similar to the ones carried out above, we see that $\tilde{\psi}$ kills the defining relations of $R$, and thus induces a ring homomorphism $\psi \colon R \to \Theta(G)$. This is clearly inverse to $\phi$, and so the proof is complete.
\end{proof}

\subsection{Measures and symbols}

We just proved that measures for $G$ are equivalent to $\Sigma$-measures. This is a significant step forward since $\Sigma$-measures are purely combinatorial objects. However, they are still rather complicated: $\Sigma$-measures have infinitely many parameters and defining equations. We now introduce $\Sigma$-symbols, which have finitely many parameters and defining equations, and connect them to $\Sigma$-measures.

\begin{definition} \label{defn:symbol}
\begin{subequations}
A \defn{$\Sigma$-symbol} with values in $k$ is a function
\begin{displaymath}
\eta \colon \Sigma^3 \to k, \qquad (a, b, c) \mapsto \eta_{a,b}(c)
\end{displaymath}
satisfying the following two conditions, for all $a,b,c,d \in \Sigma$:
\begin{align}
\eta_{a,b}(c) \eta_{c,b}(d) &= \eta_{a,b}(d) \eta_{a,d}(c) \label{eq:symb-1} \\
\eta_{a,b}(c) &= \eta_{a,d}(c)+\eta_{d,b}(c)+\delta_{c,d} \label{eq:symb-2}
\end{align}
We let $\cS_{\Sigma}(k)$ denote the set of all $\Sigma$-symbols.
\end{subequations}
\end{definition}

\begin{proposition}
We have a natural bijection $\cM_{\Sigma}(k) \to \cS_{\Sigma}(k)$ given by restricting measures to words of length~1. In other words, if $\nu$ is a $\Sigma$-measure then $(a,b,c) \mapsto \nu_{a,b}(c)$ is a $\Sigma$-symbol, and this construction is bijective.
\end{proposition}

We break the proof into several lemmas.

\begin{lemma}
Let $\nu$ be a $\Sigma$-measure and define $\eta \colon \Sigma^3 \to k$ by $\eta_{a,b}(c) = \nu_{a,b}(c)$. Then $\eta$ is a $\Sigma$-symbol, and $\nu$ can be recovered from $\eta$.
\end{lemma}

\begin{proof}
Making the substitution $(w,c,w') \to (c,d,\emptyset)$ in Definition~\ref{defn:Sigma-meas}(b), we find
\begin{displaymath}
\nu_{a,b}(c d) = \nu_{a,b}(c) \nu_{c,b}(d).
\end{displaymath}
Making the substitution $(w,c,w') \to (\emptyset,c,d)$ in the same axiom, we find
\begin{displaymath}
\nu_{a,b}(c d) = \nu_{a,b}(d) \nu_{a,b}(c).
\end{displaymath}
This gives \eqref{eq:symb-1}. Making the substition $(w,c) \to (c,d)$ in Definition~\ref{defn:Sigma-meas}(c) gives \eqref{eq:symb-2}. Thus $\eta$ is a $\Sigma$-symbol. Applying Definition~\ref{defn:Sigma-meas}(b) iteratively, we see that $\nu$ is determined by its values on length~1 words. These values are recorded by $\eta$, and so $\nu$ can be recovered from $\eta$.
\end{proof}

The above lemma provides us with an injective function $\Phi \colon \cM_{\Sigma}(k) \to \cS_{\Sigma}(k)$. To complete the proof of the proposition, we must show that $\Phi$ is surjective. Let a $\Sigma$-symbol $\eta$ be given. We recursively define $\nu$ by $\nu_{a,b}(\emptyset)=1$ and
\begin{displaymath}
\nu_{a,b}(w_1 \cdots w_n)=\eta_{a,b}(w_n) \nu_{a,w_n}(w_1\cdots w_{n-1})
\end{displaymath}
for $n \ge 1$. We clearly have $\Phi(\nu)=\eta$, provided that $\nu$ is a $\Sigma$-measure. It thus suffices to show this, which we do in the next two lemmas.

\begin{lemma}
$\nu$ satisfies Definition~\ref{defn:Sigma-meas}(b).
\end{lemma}

\begin{proof}
We must show
\begin{displaymath}
\nu_{a,b}(w c w') = \nu_{a,b}(ww') \nu_{w_r,w'_1}(c)
\end{displaymath}
for all $a,b,c \in \Sigma$ and $w,w' \in \Sigma^{\star}$, where $r=\ell(w)$; recall the convention that $w_r=a$ if $w=\emptyset$ and $w'_1=b$ if $w'=\emptyset$. We proceed by induction on the length $n$ of the word $wcw'$. The base case ($n=1$) is trivial, and the $n=2$ case follows directly from \eqref{eq:symb-1}. Suppose now that $n \ge 3$ and the identity holds in length $n-1$.

First suppose that $w'$ is non-empty, and let $s=\ell(w')$. We have
\begin{align*}
\nu_{a,b}(wc w')
&= \eta_{a,b}(w'_s) \nu_{a,w'_s} (w c w'_1 \cdots w'_{s-1}) \\
&= \eta_{a,b}(w'_s) \nu_{a,w'_s}(w w'_1 \cdots w'_{s-1}) \nu_{w_r,w'_1}(c) \\
&= \nu_{a,b}(ww') \nu_{w_r,w'_1}(c).
\end{align*}
In the first and third steps, we used the definition of $\mu$, while in the second we used the inductive hypothesis.

Now suppose that $w'$ is empty. Since $n \ge 3$, we have $r \ge 2$. We have
\begin{align*}
\nu_{a,b}(w c)
&= \nu_{a,b}(w_2 \cdots w_r c) \nu_{a,w_2}(w_1) \\
&= \nu_{a,b}(w_2 \cdots w_r) \nu_{w_r,b}(c) \nu_{a,w_2}(w_1) \\
&= \nu_{a,b}(w) \nu_{w_r,b}(c).
\end{align*}
In the first step, we applied  the previous paragraph with with $(w, \rho, w')$ being $(\emptyset, w_1, w_2 \cdots w_r c)$. In the final two steps, we used the inductive hypothesis. This completes the proof.
\end{proof}

\begin{lemma}
$\nu$ satisfies Definition~\ref{defn:Sigma-meas}(c).
\end{lemma}

\begin{proof}
For $a,b,c \in \Sigma$ and $w \in \Sigma^{\star}$, with $n=\ell(w)$, put
\begin{displaymath}
X_{a,b}^c(w) = \sum_{i=0}^n \nu_{a,c}(w[1,i]) \nu_{c,b}(w(i,n]) + \sum_{w_i=c} \nu_{a,c}(w[1,i)) \nu_{c,b}(w(i,n]).
\end{displaymath}
We must prove $X_{a,b}^c(w) = \nu_{a,b}(w)$. We proceed by induction on $n$. The case $n=0$ is clear. Thus suppose $n \ge 1$ and the identity holds for smaller $n$.

In the definition of $X_{a,b}^c(w)$, break off the $i=n$ terms from each sum, and then apply the recursive definition of $\nu$ to the second factors in the sum. This yields
\begin{align*}
X_{a,b}^c(w) =& \nu_{a,c}(w) +  \eta_{c,b}(w_n) \sum_{i=0}^{n-1} \nu_{a,c}(w[1,i]) \nu_{c,w_n}(w[i,n)) \\ &+ \delta_{w_n,c} \nu_{a,c}(w[1,n)) + \eta_{c,b}(w_n) \sum_{1 \le i \le n-1,w_i=c} \nu_{a,c}(w[1,i)) \nu_{c,w_n}(w(i,n)),
\end{align*}
and so
\begin{displaymath}
X_{a,b}^c(w) = \nu_{a,c}(w) + \delta_{w_n,c} \nu_{a,c}(w[1,n)) + \eta_{c,b}(w_n) X^c_{a,w_n}(w[1,n))
\end{displaymath}
The final $X$ factor on the right is equal to $\nu_{a,w_n}(w[1,n))$ by the inductive hypothesis. Applying the definition of $\nu$ to the first term above, and replacing $c$ with $w_n$ in the second term (which is valid due to the Kronecker delta), we thus find
\begin{align*}
X_{a,b}^c(w)
&= \eta_{a,c}(w_n) \nu_{a,w_n}(w[1,n)) + \delta_{w_n,c} \nu_{a,w_n}(w[1,n)) + \eta_{c,b}(w_n) \nu_{a,w_n}(w[1,n)) \\
&= (\eta_{a,c}(w_n) + \delta_{w_n,c} + \eta_{c,b}(w_n)) \nu_{a,w_n}(w[1,n)) \\
&= \eta_{a,b}(w_n) \nu_{a,w_n}(w[1,n)) = \nu_{a,b}(w_n).
\end{align*}
In the penultimate step, we applied \eqref{eq:symb-2}, and in the final step the definition of $\nu$. This completes the proof.
\end{proof}

\subsection{Symbols and bisection structures}

We now relate $\Sigma$-symbols to the oriented bisection structures introduced in \S \ref{ss:orbi}. We use somewhat different conventions here, though. We will take our structures valued in $\{0,1\}$ instead of $\{\pm\}$. We will also extend them by zero to the diagonal. Thus an oriented bisection structure is a function
\begin{displaymath}
S \colon \Sigma \times \Sigma \to \{0,1\}
\end{displaymath}
satisfying condition $(\ast)$ of \S \ref{ss:orbi}, and $S_{a,a}=0$ for all $a \in \Sigma$; here, and in what follows, we write $S_{a,b}$ for the value of $S$ at $(a,b)$.

\begin{proposition} \label{prop:digraph}
Let $S$ be an oriented bisection structure on $\Sigma$. Define $\eta^S \colon \Sigma^3 \to k$ by
\begin{displaymath}
\eta^S_{a,b}(c) = S_{c,a} - S_{c,b} - \delta_{b,c}
\end{displaymath}
Then $\eta^S$ is a $\Sigma$-symbol. If the ring $k$ is connected then every $\Sigma$-symbol $\eta$ has the form $\eta^S$ for a unique $S$.
\end{proposition}

Recall that $k$ is \defn{connected} if it has exactly two idempotents, namely 0 and~1; in particular, this means $1 \ne 0$ in $k$. We break the proof into several lemmas. In the first two, $S$ denotes an oriented bisection structure on $\Sigma$.

\begin{lemma}
$\eta^S$ satisfies \eqref{eq:symb-1}.
\end{lemma}

\begin{proof}
Let $a,b,c,d \in \Sigma$. We show
\begin{displaymath}
\eta^S_{a,b}(c) \eta^S_{c,b}(d) = \eta^S_{a,b}(d) \eta^S_{a,d}(c)
\end{displaymath}
There are three Kronecker $\delta$'s appearing in the above equation, namely, $\delta_{c,d}$, $\delta_{c,b}$, and $\delta_{d,b}$. We proceed in cases to handle the possible values of these.

\textit{Case 1: $b$, $c$, and $d$ are distinct.} The identity is
\begin{displaymath}
(S_{c,a}-S_{c,b})(S_{d,c}-S_{d,b}) = (S_{d,a}-S_{d,b})(S_{c,a}-S_{c,d}).
\end{displaymath}
By $(\ast)$, we have $(S_{c,d}-S_{c,b})(S_{d,c}-S_{d,b})=0$, and so
\begin{displaymath}
(S_{c,a}-S_{c,b})(S_{d,c}-S_{d,b}) = (S_{c,a}-S_{c,b})(S_{d,c}-S_{d,b})
\end{displaymath}
Similarly, we have
\begin{displaymath}
(S_{d,a}-S_{d,b})(S_{c,a}-S_{c,d}) = (S_{d,c}-S_{d,b})(S_{c,a}-S_{c,d}).
\end{displaymath}
We have thus established the identity.

\textit{Case 2: $c=d=b$.} The identity becomes
\begin{displaymath}
(-1)(S_{c,a}-1) = (S_{c,a}-1)^2,
\end{displaymath}
which is true since $S_{c,a}-1$ is either 0 or $-1$.

\textit{Case 3: $c=d$ and $c \ne b$.} The identity becomes
\begin{displaymath}
(S_{c,a}-S_{c,b})(-S_{c,b}) = (S_{c,a}-S_{c,b})(S_{c,a}-1)
\end{displaymath}
This is equivalent to
\begin{displaymath}
(S_{c,a}-S_{c,b})(S_{c,a}+S_{c,b}-1) = 0.
\end{displaymath}
If $S_{c,a}$ and $S_{c,b}$ coincide then the first factor vanishes; otherwise, one is~0 and one is~1, and so their sum is~1 and the second factor vanishes.

\textit{Case 4: $c=b$ and $c \ne d$.} Then $\eta^S_{c,b}(d)=0$, and the identity becomes
\begin{displaymath}
(S_{d,a}-S_{d,c})(S_{c,a}-S_{c,d})=0.
\end{displaymath}
This follows from $(\ast)$.

\textit{Case 5: $d=b$ and $c \ne d$.} The identity becomes
\begin{displaymath}
(S_{c,a}-S_{c,d})(S_{d,c}-1) = (S_{d,a}-1)(S_{c,a}-S_{c,d}),
\end{displaymath}
which is equivalent to
\begin{displaymath}
(S_{c,a}-S_{c,d})(S_{d,c}-S_{d,a})=0.
\end{displaymath}
This follows from $(\ast)$.
\end{proof}

\begin{lemma}
$\eta^S$ satisfies \eqref{eq:symb-2}.
\end{lemma}

\begin{proof}
Let $a,b,c,d \in \Sigma$. We have
\begin{align*}
& \eta^S_{a,d}(c)+\eta^S_{d,b}(c)+\delta_{c,d} \\
=& (S_{c,a} - S_{c,d} - \delta_{c,d}) + (S_{c,d}-S_{c,b} - \delta_{c,b}) + \delta_{c,d} \\
=& S_{c,a} - S_{c,b} - \delta_{c,b} = \eta^S_{a,b}(c).
\end{align*}
This completes the proof.
\end{proof}

\begin{lemma} \label{lem:digraph-3}
Let $\eta$ be a $\Sigma$-symbol and let $a,b,c \in \Sigma$.
\begin{enumerate}
\item We have $\eta_{a,a}(c)=-\delta_{a,c}$.
\item We have $\eta_{a,b}(c) = \eta_{c,b}(c)-\eta_{c,a}(c)-\delta_{c,a}$.
\item If $a \ne b$ then $\eta_{a,b}(a)^2 = -\eta_{a,b}(a)$.
\item If $a \ne b$ then $\eta_{c,b}(a) \eta_{c,a}(b)=0$.
\end{enumerate}
\end{lemma}

\begin{proof}
(a) By \eqref{eq:symb-2}, we have $\eta_{a,a}(c)=2\eta_{a,a}(c)+\delta_{a,c}$.

(b) By \eqref{eq:symb-2}, we have $\eta_{c,b}(c) = \eta_{c,a}(c)+\eta_{a,b}(c)+\delta_{a,c}$.

(c) By \eqref{eq:symb-1}, we have $\eta_{a,b}(a) \eta_{a,b}(a) = \eta_{a,b}(a) \eta_{a,a}(a)$, and $\eta_{a,a}(a)=-1$ by (a).

(d) By \eqref{eq:symb-1}, the product in question is $\eta_{c,b}(b) \eta_{b,b}(a)$, which vanishes by (a).
\end{proof}

\begin{lemma} \label{lem:digraph-4}
Let $\eta$ be a $\Sigma$-symbol valued in a connected ring $k$. Then there exists a unique oriented bisection structure $S$ such that $\eta=\eta^S$.
\end{lemma}

\begin{proof}
Let $a \ne b$ be elements of $\Sigma$. By Lemma~\ref{lem:digraph-3}(c), we see that $-\eta_{a,b}(a)$ is an idempotent of $k$. Since $k$ is connected, it follows that it is either~0 or~1. Define $S$ by $S_{a,b}=-\eta_{a,b}(a)$ for $a \ne b$, and $S_{a,a}=0$. Since $-\eta_{a,a}(a)=1$ by Lemma~\ref{lem:digraph-3}(a), we see that $S_{a,b}=-\eta_{a,b}(a)-\delta_{a,b}$ is valid for all $a,b \in \Sigma$.

We now verify that $S$ satisfies condition $(\ast)$. Let $a$, $b$, and $c$ be distinct elements of $\Sigma$. We must show
\begin{displaymath}
(S_{a,b}-S_{a,c})(S_{b,a}-S_{b,c})=0.
\end{displaymath}
Up to signs, the first factor is $\eta_{c,b}(a)$ and the second is $\eta_{c,a}(b)$. The product of these vanishes by Lemma~\ref{lem:digraph-3}(d), and so the claim follows. It now follows from Lemma~\ref{lem:digraph-3}(b) that $\eta=\eta^S$. Since $S$ can be recovered from $\eta^S$ (as $1 \ne 0$ in $k$), uniqueness of $S$ follows.
\end{proof}

\subsection{Proof of Theorem~\ref{mainthm}}

Let $T$ be directed $\Sigma$-labeled tree. This induces an oriented bisection structure on $\Sigma$ (\S \ref{ss:bisect-defn}) and thus, by the above results, a $\bZ$-valued measure for $G$, i.e., a ring homomorphism $\mu_T \colon \Theta(G) \to \bZ$. Consider the product of these measures
\begin{displaymath}
\phi \colon \Theta(G) \to \prod_T \bZ.
\end{displaymath}
We must show that $\phi$ is an isomorphism. For any ring $k$, there is an induced map
\begin{displaymath}
\phi^* \colon \Hom(\prod_T \bZ, k) \to \Hom(\Theta(G), k),
\end{displaymath}
where the $\Hom$'s are taken in the category of rings. By the results of this section (and Theorem~\ref{thm:bisect}), $\phi^*$ is bijective if $k$ is a connected ring. It follows that $\phi^*$ is also bijective if $k$ is a finite product of connected rings. It thus suffices to show that $\Theta(G)$ is such a product, for then Yoneda's lemma will show that $\phi$ is an isomorphism.

We now show that $\Theta(G)$ is a finite product of connected rings. It is equivalent to show that $\Theta(G)$ has finitely many idempotents, and for this it is sufficient to show that $\Theta(G)$ has finitely many minimal primes. Suppose $\fp$ is a minimal prime. Since $\Theta(G)/\fp$ is a connected ring, the quotient map $\Theta(G) \to \Theta(G)/\fp$ factors through some $\mu_T$. We thus find $\ker(\mu_T) \subset \fp$, and so $\fp=\ker(\mu_T)$ by the minimality of $\fp$. Hence $\Theta(G)$ has finitely many minimal primes.

\subsection{Description of measures}

Let $T$ be a directed $\Sigma$-labeled tree and let $\nu$ be the corresponding $\bZ$-valued $\Sigma$-measure. We now explain how to compute $\nu$ directly from $T$. Let $S$ be the oriented bisection structure associated to $T$; we use the convention that $S_{a,b}$ is~1 if $a$ points towards $b$, and~0 if $a$ points away from $b$. Let $\eta=\eta^S$ be the symbol asociated to $S$ .

Let $a,b \in \Sigma$ and let $w \in \Sigma^{\star}$ be of length $n$. Put $w_0=a$ and $w_{n+1}=b$. We say that $(a,b,w)$ is \defn{monotonic} if for all $0 \le i<j<l \le n+1$ with $w_i$, $w_j$, and $w_{\ell}$ distinct, we have that $w_j$ belongs to the shortest path joining $w_i$ and $w_l$. This means that each $w_i$ lies on the shortest path between $a$ and $b$ (inclusive), and that as we go from $w_i$ to $w_{i+1}$ we either stay at the same edge or move closer to $b$ along this path.

We say that an edge on the path from $a$ to $b$ is \defn{positively oriented} if it points away from $a$ and towards $b$; otherwise we say that it is \defn{negatively oriented}. We also apply this terminology to $a$ and $b$ themselves: $a$ is positively oriented if it points towards $b$, and $b$ is positively oriented if it points away from $a$. (If $a=b$ then $a$ is considered positively oriented.)

We say that $(a,b,w)$ is \defn{good} if it is monotonic and $w_i$ is positively oriented whenever $w_i$ occurs more than once in $w_0 \cdots w_{n+1}$. Assuming $(a,b,w)$ is good, we put $\epsilon_{a,b}(w)=(-1)^m$, where $m$ is the number of $w_i$'s, for $1 \le i \le n$, that are positively oriented. The following is our main result:

\begin{proposition} \label{prop:explicit}
If $(a,b,w)$ is good then $\nu_{a,b}(w)=\epsilon_{a,b}(w)$; otherwise $\nu_{a,b}(w)=0$.
\end{proposition}

We require a few lemmas before giving the proof.

\begin{lemma} \label{lem:explicit-1}
For $a,b,c,d \in \Sigma$, we have
\begin{displaymath}
\eta_{a,c}(b) \eta_{a,d}(c) = \eta_{a,c}(b) \eta_{b,d}(c).
\end{displaymath}
\end{lemma}

\begin{proof}
The stated equation is equivalent to
\begin{displaymath}
\eta_{a,c}(b) (\eta_{a,d}(c)-\eta_{b,d}(c)) = 0.
\end{displaymath}
We have
\begin{displaymath}
\eta_{a,d}(c)-\eta_{b,d}(c) = \eta_{a,b}(c) + \delta_{b,c},
\end{displaymath}
and so we must show
\begin{displaymath}
\eta_{a,c}(b) (\eta_{a,b}(c)+\delta_{b,c}) = 0.
\end{displaymath}
If $b \ne c$ this follows from $(\ast)$, while if $b=c$ it follows since $\eta_{a,b}(b) \in \{-1,0\}$.
\end{proof}

\begin{lemma} \label{lem:explicit-2}
For $a,b \in \sigma$ and $w \in \Sigma^{\star}$ of length $n$, we have
\begin{displaymath}
\nu_{a,b}(w) = \prod_{i=1}^n \eta_{w_{i-1},w_{i+1}}(w_i),
\end{displaymath}
where we put $w_0=a$ and $w_{n+1}=b$.
\end{lemma}

\begin{proof}
Write $\nu'_{a,b}(w)$ for the right side above. We show $\nu_{a,b}(w)=\nu'_{a,b}(w)$ by induction on the length $n$ of $w$. If $n \le 1$, the statement is clear. Now assume $n \ge 2$. We have
\begin{align*}
\nu_{a,b}(w)
&= \eta_{a,w_2}(w_1) \nu_{a,b}(w_2 \cdots w_n) \\
&= \eta_{a,w_2}(w_1) \eta_{a,w_3}(w_2) \eta_{w_2,w_4}(w_3) \cdots \\
&= \eta_{a,w_2}(w_1) \eta_{w_1,w_3}(w_2) \eta_{w_2,w_4}(w_3) \cdots = \nu'_{a,b}(w).
\end{align*}
In the first step, we used Definition~\ref{defn:Sigma-meas}(b) with $(w,c,w') \to (\emptyset, w_1, w[2,n])$; in the second step, we used the inductive hypothesis; and in the third step we applied Lemma~\ref{lem:explicit-1} to the first two factors, with $(a,b,c,d) \to (a,w_1,w_2,w_3)$. The result thus follows.
\end{proof}

\begin{lemma} \label{lem:explicit-3}
Let $a,b \in \Sigma$ be distinct. Then
\begin{displaymath}
\eta_{a,b}(a) = \begin{cases} -1 & \text{if $a$ points towards $b$} \\
0 & \text{otherwise} \end{cases} \qquad
\eta_{a,b}(b) = \begin{cases} -1 & \text{if $b$ points away from $a$} \\
0 & \text{otherwise} \end{cases}
\end{displaymath}
Let $a,b,c \in \Sigma$ be distinct. Then
\begin{displaymath}
\eta_{a,b}(c) = \begin{cases} -1 & \text{if $c$ is between $a$ and $b$, and points towards $b$} \\
1 & \text{if $c$ is between $a$ and $b$, and points towards $a$} \\
0 & \text{otherwise} \end{cases}
\end{displaymath}
\end{lemma}

\begin{proof}
These follow from direct computation. We explain the final formula. We have
\begin{displaymath}
\eta_{a,b}(c) = S_{c,a}-S_{c,b}
\end{displaymath}
since $\delta_{b,c}=0$ by assumption. The edge $c$ is between $a$ and $b$, i.e., on the shortest path joining $a$ and $b$, if and only if $S_{c,a} \ne S_{c,b}$. Thus $\eta_{a,b}(c)=0$ unless $a \le c \le b$. Assume that $a \le c \le b$. If $c$ points towards $a$ then $S_{c,a}=1$ and $S_{c,b}=0$; thus $\eta_{a,b}(c)=1$ in this case. Similarly, if $c$ points towards $b$ then $\eta_{a,b}(c)=-1$.
\end{proof}

\begin{lemma} \label{lem:explicit-4}
Let $a,b \in \Sigma$ and let $w \in \Sigma^{\star}$ have length $n$. Then for $0 \le i<j<l \le n+1$ we have that $\eta_{w_i,w_l}(w_j)$ divides $\nu_{a,b}(w)$. Here we put $w_0=a$ and $w_{n+1}=b$.
\end{lemma}

\begin{proof}
If $j=i+1$ and $l=i+2$ this follows from Definition~\ref{defn:Sigma-meas}(b) with $(w,c,w') \to (w[1,i), w_i, w(i,n])$.  Suppose now that $j>i+1$. Then by Definition~\ref{defn:Sigma-meas} with $(w,c,w') \to (w[1,i], w_{i+1}, w[i+2,n])$, we see that $\nu_{a,b}(w[1,i] w[i+2,n])$ divides $\nu_{a,b}(w)$. The former is divisible by $\eta_{w_i,w_l}(w_j)$ by induction, which completes the proof.
\end{proof}

\begin{proof}[Proof of Proposition~\ref{prop:explicit}]
Suppose $(a,b,w)$ is not good; we show that $\nu_{a,b}(w)=0$. First suppose $(a,b,w)$ is not monotonic. Then there exists $0 \le i<j<l \le n+1$ with $w_i$, $w_j$, and $w_l$ distint such that $w_j$ is not between $w_i$ and $w_l$. By Lemma~\ref{lem:explicit-3}, we have $\eta_{w_i,w_l}(w_j)=0$. By Lemma~\ref{lem:explicit-4}, this symbol divides $\nu_{a,b}(w)$, and so $\nu_{a,b}(w)=0$. Next suppose that there is a letter $c$ occurring more than once in $w_0 \cdots w_{n+1}$ that is negatively oriented. Let $i<j$ be such that $w_i=w_j=c$. If $c \ne a$ then Lemma~\ref{lem:explicit-3} shows that $\eta_{a,w_j}(w_i)=0$, while if $c \ne b$ then the same result shows that $\eta_{w_i,b}(w_j)=0$. By Lemma~\ref{lem:explicit-4}, these symbols divide $\nu_{a,b}(w)$, and so this vanishes as well.

Now suppose that $(a,b,w)$ is good. Let $1 \le i \le n$. If $w_{i-1}=w_i=w_{i+1}$ then $\eta_{w_{i-1},w_{i+1}}(w_i)=-1$. Otherwise, Lemma~\ref{lem:explicit-3} shows that $\eta_{w_{i-1},w_{i+1}}(w_i)$ is $-1$ if $w_i$ is positively oriented and $+1$ otherwise. Thus the result follows from Lemma~\ref{lem:explicit-2}.
\end{proof}

\subsection{An example} \label{ss:ex}

Let $\Sigma=\{\ra,\ldots,\rf\}$ be a six element set. Consider the following directed $\Sigma$-labeled tree $T$:
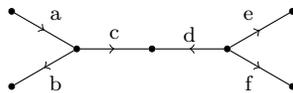
\begin{displaymath}
\begin{tikzpicture}[decoration={markings,mark=at position 0.5 with {\arrow{>}}}] 
\tikzset{leaf/.style={circle,fill=black,draw,minimum size=.75mm,inner sep=0pt}}
\node[leaf] (A) at (-1.866, .5) {};
\node[leaf] (B) at (-1.866, -.5) {};
\node[leaf] (C) at (-1,0) {};
\node[leaf] (D) at (0,0) {};
\node[leaf] (E) at (1,0) {};
\node[leaf] (F) at (1.866, .5) {};
\node[leaf] (G) at (1.866, -.5) {};
\draw[postaction={decorate}] (A) to node[align=center,xshift=1.5mm,yshift=2mm]{\tiny a} (C);
\draw[postaction={decorate}] (C) to node[align=center,xshift=1.5mm,yshift=-2mm]{\tiny b} (B);
\draw[postaction={decorate}] (C) to node[align=center,yshift=2mm]{\tiny c} (D);
\draw[postaction={decorate}] (E) to node[align=center,yshift=2mm]{\tiny d} (D);
\draw[postaction={decorate}] (E) to node[align=center,xshift=-1.5mm,yshift=2mm]{\tiny e} (F);
\draw[postaction={decorate}] (E) to node[align=center,xshift=-1.5mm,yshift=-2mm]{\tiny f} (G);
\end{tikzpicture}
\end{displaymath}
The oriented bisection structure $S$ is specified in the following table:
\begin{center}
\begin{tabular}{c|cccccc}
& a & b & c & d & e & f \\
\hline
a & 0 & 1 & 1 & 1 & 1 & 1 \\
b & 0 & 0 & 0 & 0 & 0 & 0 \\
c & 0 & 0 & 0 & 1 & 1 & 1 \\
d & 1 & 1 & 1 & 0 & 0 & 0 \\
e & 0 & 0 & 0 & 0 & 0 & 0 \\
f & 0 & 0 & 0 & 0 & 0 & 0
\end{tabular}
\end{center}
Here the row labeled ``$\ra$'' specifies the values of $S_{\ra,-}$.

Consider $\nu_{\rb,\rf}(w)$. For $(\rb,\rf,w)$ to be good, $w$ can only use the edges joining $\rb$ and $\rf$ (inclusively), i.e., $\rb$, $\rc$, $\rd$, and $\rf$, and they must appear in that order. In fact, $\rb$ can not occur in $w$ since it is negatively oriented and is one of the endpoints, and $\rd$ can only occur once. It follows that the good $w$'s in this case have the form $\rc^i \rd^j \rf^l$ where $i,l \in \bN$ and $j \in \{0,1\}$. Moreover, for such $i$, $j$, and $l$, we have
\begin{displaymath}
\nu_{\rb,\rf}(\rc^i \rd^j \rf^l) = (-1)^{i+l}.
\end{displaymath}

\section{The case of the line} \label{s:line}

Recall that $\rL=\bS \setminus \{\infty\}$ is a homogeneous structure with a total order and $\Sigma$-coloring, and its automorphism group $H$ is oligomorphic. The following theorem describes measures for $H$.

\begin{theorem}
Given a directed $\Sigma$-labeled tree $T$ and two (possibly equal) vertices $x$ and $y$ of $T$, there is an associated $\bZ$-valued measure $\mu_{T,x,y}$ for $H$. The product of these measures defines a ring isomorphism $\Theta(H) \to \prod_{T,x,y} \bZ$. In particular, $\Theta(H) \cong \bZ^M$, where $M=(2n+2)^n$ and $n=\# \Sigma$.
\end{theorem}

\begin{proof}
We simply indicate the main ideas of the proof. Put $\ol{\Sigma}=\Sigma \cup \{\pm \infty\}$. A \defn{$\ol{\Sigma}$-measure} is a rule assigning to $a \in \Sigma \cup \{-\infty\}$, $b \in \Sigma \cup \{+\infty\}$, and $w \in \Sigma^{\star}$ a value $\nu_{a,b}(w)$ in $k$ satisfying axioms similar to those in Definition~\ref{defn:Sigma-meas}. One first shows that measures for $H$ are equivalent to $\ol{\Sigma}$-measures: given a measure $\mu$ for $H$, the corresponding $\ol{\Sigma}$-measure is defined by $\nu_{a,b}(w)=\mu(I^w)$, where $I$ is an interval in $\bL$ with endpoints of type $a$ and $b$.

Next, a \defn{$\ol{\Sigma}$-symbol} is a rule assigning to $a \in \Sigma \cup \{-\infty\}$, $b \in \Sigma \cup \{+\infty\}$, and $c \in \Sigma$ a value $\eta_{a,b}(c)$ in $k$ satisfying axioms similar to those in Definition~\ref{defn:symbol}. One shows that there is a bijective correspondence between $\ol{\Sigma}$-measures and $\ol{\Sigma}$-symbols.

Suppose $\eta$ is a $\ol{\Sigma}$-symbol valued in a connected ring $k$. Define
\begin{displaymath}
S \colon \Sigma \times \ol{\Sigma} \to \{0,1\}
\end{displaymath}
as follows. For $a,b \in \Sigma$, we define $S_{a,b}$ just as before. We put $S_{a,\infty}=-\eta_{a,\infty}(a)$ and $S_{a,-\infty}(a)=\eta_{-\infty,a}(a)+1$. One shows using arguments similar to before that these values do belong to $\{0,1\}$, and that $\eta$ can be recovered from $S$. Moreover, $S$ satisfies the following condition: given $a,b \in \Sigma$ and $c \in \ol{\Sigma}$, at least one of the equalities $S_{a,b}=S_{a,c}$ or $S_{b,a}=S_{b,c}$ holds.

Restricting $S$ to $\Sigma \times \Sigma$ yields an oriented bisection structure, and thus a directed $\Sigma$-labeled tree $T$. One then shows that there is a unique vertex $x$ of $T$ such that $S_{a,\infty}$ is~1 if $a$ points to $x$, and~0 otherwise; similarly, one gets a vertex $y$ associated to $-\infty$. From this description, it is also clear how one can start with $(T,x,y)$ and then define $S$, $\eta$, $\nu$, and finally $\mu$. This is how one obtains the stated description for $\Theta(H)$.

Finally, we explain the enumeration. If $T$ is a tree with $n$ edges then $T$ has $n+1$ vertices, and so there are $(n+1)^2$ choices for $(x,y)$. Hence $M=(n+1)^2 \cdot N$, where $N$ is the number of choices for $T$. (Note that directed $\Sigma$-labeled trees have no non-trivial automorphisms.) We have already seen in Proposition~\ref{prop:enum} that $N=2^n \cdot (n+1)^{n-2}$.
\end{proof}

\end{document}